\documentclass[12pt,reqno]{amsart}
\usepackage{amsthm,amsfonts,amssymb,euscript}
\usepackage{graphics}
\usepackage{graphicx}

\newtheorem{theorem}{Theorem}[section]
\newtheorem{lemma}[theorem]{Lemma}

\newtheorem{proposition}[theorem]{Proposition}

\newtheorem{definition}[theorem]{Definition}
\newtheorem{remark}[theorem]{Remark}

\numberwithin{equation}{section}

\begin{document}\title[Uniqueness]{A uniqueness theorem for free waves on $\mathbb{R}^{n+1}$}

\author{Phillip Whitman}
\address{Department of Mathematics\\
Princeton University\\
USA}
\email{pwhiman@math.princeton.edu}

\author{Pin Yu}
\address{Mathematical Sciences Center\\
Tsinghua University\\
China}
\email{pinyu@math.princeton.edu}
\maketitle

\begin{abstract}
In this short note, based on Carleman estimates and Holmgren's type theorems, we provide a converse theorem of the classical Huygens principle for free wave equations on $\mathbb{R}^{n+1}$. Possible generalizations to other underlying space-times or other wave type equations are also discussed.
\end{abstract}

\section{Introduction}
We shall study uniqueness properties of free linear wave equation
\begin{equation}\label{free_wave}
 \Box \varphi = 0,
\end{equation}
on $\mathbb{R}^{n+1}$ where we assume $\varphi$ is a smooth (although $C^2$ is enough, we shall not stick to the minimal regularity) solution. In view of the Huygens principle or standard energy estimates, if the initial data vanishes on a set on $\{ t=0 \}$ (here vanishing means that both $\varphi$ and $\partial_t \varphi$ vanish), the free wave vanishes on the domain of dependence of this set. In this note, we provide a converse version of this statement. To state one version of the main theorem,  we first list some notations. Let $r=\sqrt{\sum_{i=1}^n x_i^2}$ be the standard radius function. We use $B_{t_0}(r) \subset \{t=t_0\} $
to denote the $n$-dimensional ball of radius $r$ centered at the origin of hyperplane $\{t=t_0\}$.
The Huygens principle predicts that if the data $(\varphi, \partial_t \varphi)|_{t=0} = 0$ on the ball $B_{0}(3)$, then $(\varphi, \partial_t \varphi)$ vanishes on $B_{1}(2)$ and $B_{-1}(2)$. In converse, if one knows  $(\varphi, \partial_t \varphi)$ vanishes on $B_{1}(2)$ and $B_{-1}(2)$, by the Huygens principle, one can only say that at time slice $t=0$, the wave must vanish on $B_0(1)$. One may ask if it is possible to show more, say to determine the maximal domain on which the wave vanishes. The answer is yes: the following theorem shows that the knowledge of free waves on $B_{1}(2)$ and $B_{-1}(2)$ are enough to determine the waves on $B_0(3)$:

\begin{theorem}\label{theorem1}
Assume that $\varphi$ is a smooth solution
\begin{equation}
 \Box \varphi = 0,
\end{equation}
if $(\varphi, \partial_t \varphi)$ vanishes on $B_{1}(1)$ and $B_{-1}(1)$, then it must vanish on $B_0(3)$.
\end{theorem}

If we study waves on $\mathbb{R}^{1+1}$, the above theorem is almost obvious: we can decompose the wave into outgoing and incoming components and show that each component is determined by its value on $B_{1}(2)$ and $B_{-1}(2)$. This proof also motivates the theorem on higher dimensions. We define the standard optical functions $u$ and $v$ on the Minkowski space-time:
\begin{equation*}
 u = r + t, \quad v = r - t,
\end{equation*}
Let $\mathcal{H}$ be the bifurcate light cone
\begin{equation*}
 \mathcal{H} = \{(t,x) | r = |t| + 1\},
\end{equation*}
where $x = (x_1,\cdots, x_n)$ and let $\mathcal{H_\varepsilon}$ be the truncated light cone
\begin{equation*}
 \mathcal{H}_{\varepsilon, L} = \{(t,x) | r = |t| + 1, -\varepsilon < t < L\},
\end{equation*}
where $0<\varepsilon \leq L$. We use $\mathcal{D}_{\varepsilon,L}$ to denote the following set
\begin{equation*}
 \{(t, x)| 1<u<2L+1, 1<v<1+2\varepsilon\}.
\end{equation*}
Our main theorem is the following uniqueness property for characteristic problems:
\begin{theorem}\label{MainTheorem}
 If $\varphi$ is a $C^2$ solution for the free wave equation
\begin{equation*}
 \Box \varphi = 0,
\end{equation*}
and $\varphi \equiv 0$ on $\mathcal{H}_{\varepsilon, L}$, then $\varphi \equiv 0$ on $\mathcal{D}_{\varepsilon,L}$.
\end{theorem}
Theorem \ref{theorem1} is an easy consequence of Theorem \ref{MainTheorem} by setting the parameters $\varepsilon = L =1$. In fact, Theorem \ref{MainTheorem} can also be deduced from the proof of Theorem \ref{theorem1} which is given in the next section.

We now discuss briefly the related history of uniqueness problems. The question of uniqueness of smooth solutions to linear wave equations with Cauchy data prescribed on a smooth hypersurface is very strongly linked to whether you demand your coefficients, surface and solution to be analytic or merely infinitely differentiable.  In the case that the coefficients, the surface and the solution are analytic, then it is known that unique continuation (and in fact existence) holds across any non-characteristic surface. This is known as the Cauchy-Kovalewski theorem. If the coefficients and the surface are analytic but the solution is merely infinitely differentiable, the same uniqueness (though not existence) result is true by Holmgren's theorem (see \cite{John} for proofs of Cauchy-Kovalewski and Holmgren's theorems). When the coefficients and the solution are infinitely differentiable but not necessarily analytic, unique continuation does not in general hold for non-characteristic surfaces (see \cite{Alinhac-Baouendi} and \cite{Hormander} for counterexamples). It is in this setting that the notion of pseudoconvexity arises. This condition says that for unique continuation to hold across a surface, characteristic curves tangent to the surface must bend toward the side of the surface where you have information. In the wave equation case, this condition is equivalent to the following:
\begin{definition}
We say the surface, S, given by $\{(t,x)|f(t,x)=0\}$ is \emph{strongly pseudoconvex} if for every vector field, $X$, we have
\begin{equation}
g(X,X)=g(X,\nabla f)=0\Rightarrow \nabla^2 f(X,X)>0,
\end{equation}
where the gradient $\nabla$ and Hessian $\nabla^2$ are taken with respect to the standard Minkowski metric.
\end{definition}
In this case a theorem of H\"ormander (in fact in the more general setting) shows that pseudoconvexity implies unique continuation. This is done through the use of Carleman estimates. What about in between?  Namely, are analyticity of coefficients or the solution in certain variables and not in others enough for uniqueness to hold?  The first theorem in this direction was the paper of Lerner \cite{Lerner}. He proves a global uniqueness theorem for equations of the form $\partial_t^2-A(x,D_x)$ where $A$ is uniformly elliptic across the surface $x_1=0$ under some assumptions (in order to prove energy estimates to control the solution as $t\rightarrow\infty$). Once he gains this control, he takes a Fourier-Gauss transform to turn the equation into an elliptic problem, he then uses known estimates to prove uniqueness for the elliptic equation and then inverts the Fourier-Gauss transform to conclude. Later, Robbiano \cite{Robbiano} was able to prove a local theorem of this type. Eventually Tataru \cite{Tataru01} and then Robbiano-Zuily \cite{Robbiano-Zuily} were able to prove this type of theorem for any time-like surface. In fact they were able to do much more. First of all the coefficients no longer had to be independent of $t$, they could be merely analytic in $t$. Secondly, it was done for arbitrary linear partial differential operators, not just wave equations. Finally, they were able to extend the theorem to situations when $\partial_t$ was not time-like, and instead the operator obeyed a different pseudoconvexity condition. This pseudoconvexity condition has the following form:
\begin{definition}
Let $T$ be a Killing vector field, then we say that a surface $S=\{f=0\}$ is $T$-conditional pseudoconvex at a point $x_0 \in S$ if for every vector field, $X$, we have that $g(X,X)=g(X,T)=g(X,\nabla f)=0 \Rightarrow \nabla^2 f(X,X)<0$ at the point $x_0$
\end{definition}
If the Killing vector field $T$ is time-like, then the condition is satisfied trivially, because there are no null vectors orthogonal to $T$. In our case we take the vector field $T=\partial_t$ which is everywhere time-like. We can now state the theorem proven in \cite{Tataru01} and \cite{Robbiano-Zuily} in the wave equation setting:
\begin{proposition}\label{local_uniqueness_theorem}
Let $S$ be a surface in Minkowski space locally given by the graph of a function $\{f(x)=0\}$ near a point $(t_0,x_0)$ (say in a neighborhood $U$ of $(t_0,x_0)$). Let $\varphi$ be zero in $U \cap \{f \leq 0\}$, and solve a linear wave equation
\begin{equation}
\Box \varphi=V\varphi + W(\varphi),
\end{equation}
where $V$ is a $C^\infty$ function and $W$ is a $C^\infty$ vector field. Assume that $V$ and $W$ are analytic in the variable $t$. Then there exists a neighborhood $\tilde{U}$ of $(t_0,x_0)$ so that $\varphi=0$ in $\tilde{U}$
\end{proposition}
In applications of the theorem in this paper, we shall take $V=0$ and $W=0$, but the above theorem clearly suggests the way to generalize our results to other types of linear wave equations (namely, assuming analyticity in a time-like variable of the coefficients).

Recently, in order to prove global rigidity for stationary black holes, Ionescu and Klainerman in \cite{Ionescu-Klainerman01} and \cite{Ionescu-Klainerman02} considered the uniqueness problem across null characteristic surfaces. To obtain uniqueness, they proved a Carleman type estimate. In the next section, we will adapt their proof as a first step towards the proof of Theorem \ref{MainTheorem}.

\section{Proof of Theorem \ref{theorem1}}
\subsection{Carleman Estimates}
The main ingredient of the proof of Theorem \ref{theorem1} is the following Carleman type estimate due to Ionescu and Klainerman \cite{Ionescu-Klainerman01},
\begin{proposition}
 Given three positive numbers $l<\tilde{l}$ and $R$, we define the weight function
\begin{equation}
 f_l(u,v) = \log((u-l)(v-l))
\end{equation}
and the region $\Omega_{(l,\tilde{l})}(R)$ in $\mathbb{R}^{n+1}$
\begin{equation*}
\Omega_{(l,\tilde{l})}(R) =  \{(t,x) \in \mathbb{R}^{n+1} | u> \tilde{l}, v>\tilde{l}\} \cap  \{(t,x) \in \mathbb{R}^{n+1} | (u-l)(v-l) < R\}
\end{equation*}
Then there exists a constant $\lambda(l, \tilde{l}, R)$ and $C(l, \tilde{l}, R)$ such that for any $\varphi \in C^2_0(\Omega_{(l,\tilde{l})}(R))$ and any $\lambda > \lambda(l, \tilde{l}, R)$, we have the following estimates
\begin{equation}\label{Carleman_Estimates}
\lambda^{\frac{3}{2}} \|e^{-\lambda f_l} \cdot \varphi\|_{L^2} + \sum_{\alpha = 0}^{n}\lambda^{\frac{1}{2}}\|e^{-\lambda f_l} \cdot \partial_\alpha \varphi\|_{L^2} \leq C(l, \tilde{l}, R) \|e^{-\lambda f_l} \cdot \Box \varphi\|_{L^2}
\end{equation}
\end{proposition}

\subsection{Vanishing on the First Neighborhood $\mathcal{O}$}
To apply the Carleman estimates, we first construct a cut-off function $\eta_{\varepsilon, \delta}(u,v)$. Let $\chi_{+}(x)$ be a smooth function defined on $\mathbb{R}$ supported on $\mathbb{R}^+$ and equal to $1$ on $[1,\infty)$. Let $\chi_{-}(x) = \chi_{+}(-x)$. We will take $\tilde{l} =1$ and take a real number $l \in (0,1)$. Recall that
\begin{equation*}
 \mathcal{D}=\mathcal{D}_{1,1}=\{(t, x)| 1<u<3, 1<v<3\},
\end{equation*}
we also define the first neighborhood to be
\begin{equation*}
\mathcal{O} = \{(t,x) \in \mathcal{D}| (u-l)(v-l)<(3-l)(1-l)\}.
\end{equation*}
The size of $l$ will determine the size of the first neighborhood. In fact, one can take $l$ as small as we want, this will enlarge the first neighborhood. We define
\begin{equation}
 \eta_{\varepsilon, \delta}(u,v) = \chi_{+}(\frac{(u-1)(v-1)}{\varepsilon})\chi_{-}(\frac{(u-l)(v-l)-(3-l)(1-l)}{\delta})
\end{equation}
We apply Carleman estimates \eqref{Carleman_Estimates} to the function $\eta_{\varepsilon, \delta}(u,v)\varphi(x)$, by ignoring the derivatives on the left hand side (these terms are useful when one deals with non-homogenous wave equations), we have
\begin{align}\label{estimate_1}
\lambda^{\frac{3}{2}} \|e^{-\lambda f_l} \cdot \eta_{\varepsilon, \delta} \cdot \varphi\|_{L^2} &\lesssim \|e^{-\lambda f_l} \cdot \Box(\eta_{\varepsilon, \delta} \cdot \varphi)\|_{L^2} \notag\\
&\lesssim \|e^{-\lambda f_l} \cdot  \varphi \cdot \Box\eta_{\varepsilon, \delta} \|_{L^2} + \|e^{-\lambda f_l} \cdot \nabla^\alpha \eta_{\varepsilon, \delta} \cdot \nabla_\alpha \varphi\|_{L^2}
\end{align}
where the repeated indices are understood as subject to the Einstein convention with respect to the standard Lorentzian metric on Minkowski spaces.
On the region $$\mathcal{D} = \{(u,v)| u-1 \in [0,1) \quad \text{and} \quad  v-1 \in [0,1)\},$$ since $\varphi \in C^2(\mathbb{R}^{3+1})$ vanishes on the boundary $\mathcal{H}$ of $\mathcal{D}$ (this is a consequence of Huygens principle. We remark that, in fact, to show that $\varphi$ vanishes on $\mathcal{D}$, we only need to assume $\varphi$ vanishes on $\mathcal{H}$), we know that there is a function $\psi$ such that
\begin{equation}
  \varphi = (u-1)(v-1) \psi.
\end{equation}
We now show that
\begin{equation}
 |\Box \eta_{\varepsilon, \delta}| \lesssim_{\delta} \frac{1}{\varepsilon} C.
\end{equation}
Since $\eta_{\varepsilon, \delta}$ depends only on $(u,v)$, then $\Box = -\partial_u \partial_v$ as highest order derivatives on $\eta_{\varepsilon, \delta}$. So the worst possible scenario is
\begin{align*}
 \partial_u \partial_v \eta_{\varepsilon, \delta} &= \partial_u \partial_v [  \chi_{+}(\frac{(u-1)(v-1)}{\varepsilon})\chi_{-}(\frac{(u-l)(v-l)-(3-l)(1-l)}{\delta})]\\
&=\chi_{-} \partial_u \partial_v [\chi_{+}(\frac{(u-1)(v-1)}{\varepsilon})] + O_{\delta} (\frac{1}{\varepsilon})\\
&=\frac{(u-1)(v-1)}{\varepsilon^2} \chi_{+}''\chi_{-}  + O_{\delta} (\frac{1}{\varepsilon})
\end{align*}
Since $\chi_{+}''$ is supported in the region where $|u-1||v-1| \leq \varepsilon$, this gives the desired estimates.
We turn to the following estimates
\begin{align*}
 \nabla^\alpha \eta_{\varepsilon, \delta} \cdot \nabla_\alpha \varphi &= -\frac{1}{2}\partial_u \varphi \partial_v \eta_{\varepsilon, \delta}- \frac{1}{2}\partial_u \eta_{\varepsilon, \delta} \partial_v \varphi = \frac{(u-1)(v-1)}{\varepsilon} \psi \cdot \chi_{+}''\chi_{-} + O_\delta(1)
&\lesssim O_\delta(1)
\end{align*}
By putting all the estimates together, we find that terms in \eqref{estimate_1} are independent of $\varepsilon$. This allows us to take $\varepsilon \rightarrow 0$. Now we can modify the cut-off function to be
\begin{equation}
 \eta_{\delta}(u,v) = \chi_{-}(\frac{(u-l)(v-l)-(3-l)(1-l)}{\delta})
\end{equation}
with the estimates
\begin{align}\label{estimate_2}
\lambda^{\frac{3}{2}} \|e^{-\lambda f_l} \cdot \eta_{\delta} \cdot \varphi\|_{L^2} &\lesssim \|e^{-\lambda f_l} \cdot  \varphi \cdot \Box\eta_{\delta} \|_{L^2} + \|e^{-\lambda f_l} \cdot \nabla^\alpha \eta_{\delta} \cdot \nabla_\alpha \varphi\|_{L^2} \notag\\
&\lesssim_\delta  \|e^{-\lambda f_l}  \|_{{L^2}(\mathcal{D}_\delta)}
\end{align}
where $\mathcal{D}_\delta = \{x \in \mathcal{D} | \nabla \eta_\delta(x) \neq 0\}$. Notice that the region $\mathcal{D}^1_\delta = \{x \in \mathcal{D} | \eta_\delta(x) = 1\}$ is contained in $\mathcal{D} - \mathcal{D}_\delta$, and most importantly, we have
\begin{equation}
 \inf_{\mathcal{D}^1_\delta} e^{-\lambda f_l} \geq \sup_{\mathcal{D}_\delta}  e^{-\lambda f_l}
\end{equation}
Back to \eqref{estimate_2}, we have
\begin{align*}
\lambda^{\frac{3}{2}} \inf_{\mathcal{D}^1_\delta} e^{-\lambda f_l} \| \varphi\|_{L^2(\mathcal{D}^1_\delta)} &\lesssim \lambda^{\frac{3}{2}} \|e^{-\lambda f_l} \cdot \eta_{\delta} \cdot \varphi\|_{L^2(\mathcal{D})} \lesssim_\delta  \|e^{-\lambda f_l}  \|_{{L^2}(\mathcal{D}_\delta)} \lesssim \sup_{\mathcal{D}_\delta}  e^{-\lambda f_l}  \|1 \|_{{L^2}(\mathcal{D}_\delta)}
\end{align*}
Let $\lambda \rightarrow \infty$, we have $\|\eta_{\delta} \cdot \varphi\|_{L^2(\mathcal{D}^1_\delta)}=0$, which implies
\begin{equation}
 \varphi = 0
\end{equation}
on $\mathcal{D}^1_\delta$ and let $\delta \rightarrow 0$, we know that $\varphi$ vanishes on the first neighborhood $\mathcal{O}$.

\subsection{Vanishing on the full neighborhood $\mathcal{D}$}
In the previous subsection, thanks to Carleman estimates adapted to the bifurcate null hypersurface $\mathcal{H}$, we are able to show that $\varphi$ vanishes on $\mathcal{O}$. We turn to the proof that $\varphi$ also vanishes on $\mathcal{O}^c = \mathcal{D}-\mathcal{O}$.

For a given small parameter $\delta >0$, we first define slightly smaller regions $\mathcal{D}(\delta)$ and $\mathcal{O}^c(\delta)$ as follows,
\begin{align*}
 \mathcal{D}(\delta) &= \{(t,x) \in \mathcal{D} | 1 < u < 3-\delta, 1 < v < 3-\delta \},\\
 \mathcal{O}^c(\delta) &=\{(t,x) \in \mathcal{O}^c | 1 < u < 3-\delta, 1 < v < 3-\delta \}.
\end{align*}
Thus it suffices to show that $\varphi$ vanishes on $\mathcal{O}^c(\delta)$ for all $\delta > 0$. The key point of the proof concerns the geometry of the region $\mathcal{D}(\delta)$,
\begin{lemma}
 $\mathcal{D}(\delta)$ can be foliated by a family of time-like hypersurfaces $H_s$ parameterized by $s \in \mathbb{R}^+$.
\end{lemma}
\begin{proof}
We shall exhibit explicitly such a family. We define a function
\begin{equation*}
 H(t,x) = \frac{u-1}{3-\delta-u} \cdot \frac{v-1}{3-\delta-v}.
\end{equation*}
Thus for $s \in \mathbb{R}^+$, it is easy to show that the level surfaces $H_s = H^{-1}(s)$ are time-like thus they give the desired foliation.
\end{proof}
We shall show that for each $s \in \mathbb{R}$, $\varphi$ vanishes on $H_s$. Notice that if $s$ is small enough,
\begin{equation*}
 H_s \subset \mathcal{O} \cap \mathcal{D}(\delta).
\end{equation*}
In view of the fact that $\varphi \equiv 0$ on the first neighborhood $\mathcal{O}$, $\varphi$ vanishes on $H_s$ for small $s$. We now use continuity argument to extend such $s$ to the whole parameter space $\mathbb{R}^+$. It suffices to show that, if $\varphi$ vanishes on $H_s$ for $s \leq s_0$, thus there is a $\varepsilon >0$, such that $\varphi$ vanishes on $H_s$ for $s \leq s_0 + \varepsilon$. On $H_{s_0}$, we can apply Proposition \ref{local_uniqueness_theorem} to conclude that for each $p \in H_{s_0}$, there is a neighborhood $U(p)$ of $p$ in $\mathcal{D}(\delta)$ such that $\varphi \equiv 0$ on $U(p)$. Since $\varphi$ is already shown to be zero on $\mathcal{O}$, thus we may concentrate on the region $\mathcal{O}^c(\delta)$. We observe that $H_{s_0} \cap \mathcal{O}^c(\delta)$ is compact relative to the reduced topology from $\mathcal{D}(\delta)$, thus we can use finite many $U(p)$'s to cover $H_{s_0}$. Obviously, this cover allows us to extend the range of $s$ to $(0,s_0 + \varepsilon)$ for some $\varepsilon>0$.

The previous argument shows that $\varphi \equiv 0$ on $\mathcal{D}(\delta)$. By taking $\delta \rightarrow 0$, we conclude that $\varphi \equiv 0$ on the entire $\mathcal{D}$. According to Huygens principle, this is the maximal domain in which one expects $\varphi$ to vanish. This completes the proof of Theorem \ref{theorem1}.

\begin{remark}\label{remark_bifurcate_picture}
In this section, to show that $\varphi \equiv 0$ on $\mathcal{D}$, we only require $\varphi$ vanishes on the bifurcate null hypersurface
\begin{equation*}
 \{(t,x)| 0 \leq t\leq 1, v = 1\} \cup \{(t,x)| -1 \leq t \leq 0, u = 1\}.
\end{equation*}
\end{remark}

\section{Proof of Theorem \ref{MainTheorem}}
In view of the proof of Theorem \ref{theorem1} and Remark \ref{remark_bifurcate_picture}, it is easy to see that $\varphi \equiv 0$ on $\mathcal{D}_{\varepsilon,\varepsilon}$. We shall repeat this argument. Let $L = k \varepsilon + \varepsilon'$ where $k \in \mathbb{N}$ and $0 \leq \varepsilon' < \varepsilon$. We define a sequence of regions $\mathcal{D}_j$ for $j=1, 2, \cdots, k$ as follows,
\begin{equation*}
 \mathcal{D}_j = \{(t,x) \in \mathcal{D}| 1+2 (j-1)\varepsilon \leq u \leq 1+2 j \varepsilon \}.
\end{equation*}
We also define $\mathcal{D}' = \mathcal{D} - \cup_{j=1}^k \mathcal{D}_j$. Notice that $\mathcal{D}_{1} = \mathcal{D}_{\varepsilon,\varepsilon}$, thus $\varphi \equiv 0$ on $\mathcal{D}_1$. We show that $\varphi \equiv 0$ on each $\mathcal{D}_j$ where $1\leq j \leq k$. In fact, if $\varphi \equiv 0$ for $\mathcal{D}_{j}$, in view of the time translation invariance of the wave equation, we can obviously use the proof of Theorem \ref{theorem1} and Remark \ref{remark_bifurcate_picture} to conclude that $\varphi \equiv 0$ on $\mathcal{D}_{j+1}$.

For $\mathcal{D}'$ we can repeat the above argument. An easy limit argument completes the proof of Theorem \ref{MainTheorem}.

\begin{remark}
There are many ways to generalize Theorem \ref{MainTheorem}. If one still works on Minkowski space, we can allow lower order terms as we have in Proposition \ref{local_uniqueness_theorem} and we can also work with a system of equations instead of a single one. The strategy of the proof is the same. For the existence of the first neighborhood, we still use the Carleman type estimates of \emph{Ionescu} and \emph{Klainerman}; for the extension to the whole neighborhood, the argument remains the same. If one wants to generalize to other space-times, we notice that once we can construct optical functions $u$ and $v$ the argument for the existence of the first neighborhood is stable. This is shown in \cite{Ionescu-Klainerman01} and \cite{Ionescu-Klainerman02}. We only require the space-time is smooth. While for extensions to the whole neighborhood, we want to apply Proposition \ref{local_uniqueness_theorem} and this may require the analyticity of the space-time and the hypersurfaces $H_s$. Nonetheless, this theorem holds for Schwarzschild space-times or Kerr Space-times.
\end{remark}

\end{document}